\documentclass[reqno,10pt]{amsart}

\usepackage{amsmath,amsfonts,amssymb,amsthm,epsfig,amscd,}
\usepackage[]{fontenc}
\usepackage[latin1]{inputenc}
\usepackage{enumerate}
\usepackage[cmtip,all]{xy}
\usepackage{tabularx,multirow}
\usepackage{color}
\usepackage{url}

 \allowdisplaybreaks \numberwithin{equation}{section}

\newtheorem{theorem}{Theorem}[section]

\newtheorem{proposition}[theorem]{Proposition}

\newtheorem{lemma}[theorem]{Lemma}

\newtheorem{claim}[theorem]{Claim}

\theoremstyle{definition}
\newtheorem{definition}[theorem]{Definition}

\theoremstyle{remark}
\newtheorem{remark}[theorem]{Remark}

\newcommand{\FF}{\mathbb{F}}
\newcommand{\Pp}{\mathbb{P}}
\newcommand{\Oc}{\mathcal{O}}

\def\leq{{\leqslant}}
\def\geq{{\geqslant}}

\begin{document}

\title[Big Vector Bundles on Surfaces and Fourfolds]{Big Vector Bundles on Surfaces and Fourfolds}

\thanks{This research started during the visit of both authors at the {\em Laboratoire de Math\'ematiques et Applications}--Universit\'e de Poitiers; the authors thank 
this institution for the excellent hospitality and the friendly atmosphere. The collaboration has been partially supported by the
Research Project \emph{``Families of curves: their moduli and their related varieties"} (CUP: E81-18000100005) - Mission Sustainability - University of Rome Tor Vergata. 
Finally, the authors would like to thank the anonymous referee for carefully reading the manuscript, for his/her comments on the first version of this work (which have improved the exposition of the paper) and for valuable remarks inspiring Examples \ref{ExSurf}-(e) and \ref{Ex4fold}-(b) of the present version.\color{black}} 

\author{Gilberto Bini}
\address{Gilberto Bini, Dipartimento  di Matematica ``F. Enriques", Universit\`{a} degli Studi di Milano, Via C. Saldini 50, 20133 Milano -- Italy}
\email{gilberto.bini@unimi.it}

\author{Flaminio Flamini}
\address{Flaminio Flamini, Dipartimento  di Matematica, Universit\`{a} degli Studi di Roma ``Tor Vergata", Viale della Ricerca Scientifica 1, 00133 Roma -- Italy}
\email{flamini@mat.uniroma2.it}

\begin{abstract} The aim of this note is to exhibit explicit sufficient cohomological criteria ensuring bigness of globally generated, rank-$r$ vector bundles, $r \geqslant 2$, 
on smooth, projective varieties of even dimension $d \leqslant 4$. We also discuss connections of our general criteria to some recent results of other authors, as well as 
applications to tangent bundles of Fano varieties, to suitable Lazarsfeld-Mukai bundles on four-folds, etcetera.  
\end{abstract}

\maketitle

\section{Introduction}

Let $V$ be a smooth projective variety over the field of complex numbers, and denote by $d$ the dimension of it. As well known, the geometry of $V$ can be described by way of linear systems of divisors $D$ on $V$. The resulting mapping from $V$ to projective space has different characteristics according to the positivity of ${\mathcal O}_V(D)$: see \cite{Laz1} for a comprehensive treatment of these topics. In particular, we recall that a divisor $D$ is {\em big} if and only if the {\em Kodaira-Iitaka dimension} of ${\mathcal O}_V(D)$ is equal to $\dim(V)$. More geometrically, the Iitaka fibration theorem implies that ${\mathcal O}_V(D)$ is big if and only if the mapping $\phi_m:V \dashrightarrow {\mathbb P}H^0\left(V,{\mathcal O}_V(mD)\right)^{\vee}$ is birational onto its image for some $m>0$: see, for instance, \cite[p.\;139]{Laz1}. Moreover, just to mention a few results, there are cohomological and numerical criteria for a divisor to be big. Remarkably, a globally generated line bundle - hence nef - is big if and only if the top self intersection $c_1({\mathcal O}_V(D))^d$ is positive (cf.\;e.g.\;\cite[Thm.\;2.2.16]{Laz1}). 

As recalled in the Introduction to \cite{Laz2}, in the past 60 years there has been a considerable effort to generalize the theory of positivity of line bundles to vector bundles, in particular to extend the cohomological and topological properties of ample divisors. In this paper, we will focus on some aspects of the whole theory, which is rather articulated: the reader may find a recent exposition in \cite{GG}, where the various notions of positivity for vector bundles are studied in connection with topics from Hodge Theory, Satake-Baily-Borel completion of period mappings, Iitaka conjecture, etcetera. Also, positivity of vector bundles, especially the tangent bundle $T_V$, is related to the classification of projective manifolds: see, for instance, \cite{CP} and, for the more general K\"ahler manifolds, \cite{DPS}.

We recall that a rank $r \geq 2$ vector bundle $E$ on $V$ is ample (nef) if the tautological bundle ${\mathcal O}_{{\mathbb P}(E)}(1)$ of the projective bundle $\pi: {\mathbb P}(E) \rightarrow V$ is an ample (nef) line bundle. As for the notion of bigness, there are various definitions: see, for instance, \cite{BK} for them and their relation to base loci of vector bundles. Here we will deal with the notion of {\em $L$-bigness}, i.e., a vector bundle $E$ is $L$-big if and only if the tautological bundle of ${\mathbb P}(E)$ is a big line bundle (cf.\cite[(6.1.2) in Def.\;6.1]{BK}). In what follows, we will drop the $L$ and simply talk about {\em big vector bundles}. As in the case of line bundles, bigness of vector bundles has a geometric interpretation in terms of birational images of the ruled variety ${\mathbb P}(E)$ in suitable projective spaces.

In this paper, our aim is to investigate natural cohomological conditions for a globally generated vector bundle $E$ to be big on $V$. Roughly speaking, this is our strategy. Since a globally generated vector bundle is nef (see, for the sake of completeness, Remark \ref{rem:numerical}), \cite[Theorem\;2.5]{DPS} implies that the nefness of $E$ can be measured in terms of the non-negativity of $(-1)^ds_d(E)$, where $s_d(E)$ is the top Segre class of $E$. What's more, a nef vector bundle has a well-defined {\em numerical dimension} $n(E)$, which is the numerical dimension of the tautological bundle ${\mathcal O}_{{\mathbb P}(E)}(1)$, i.e., the largest non-negative integer $n(E)$ such that $
c_1\left({\mathcal O}_{{\mathbb P}(E)}(1)\right)^{n(E)}$
is not numerically equivalent to $0$ (cf. Def. \ref{def:numdim2} below). 

If the $d^{th}$ Segre class of $E$ is positive, which restricts our investigation to $d$ even, one can see that the numerical dimension $n(E)$ equals the dimension of ${\mathbb P}(E)$, which in turn means that the tautological bundle on $V$ is a big line bundle, so $E$ is big. 

\vskip 30pt

With this setting, here are our results.

\vskip 7pt

\noindent
{\bf Theorem} (cf. Theorem \ref{thm:surf} below) {\em Let $V$ be any smooth, irreducible projective surface. 
Let $E$ be a globally--generated, rank--$r$ vector bundle on $V$, $r \geqslant 2$, such that $h^0(E) \geqslant r+2$. 
Assume further that $h^1((\det E)^{-1}) = 0$. Then $E$ is a big vector bundle on $V$.}

\vskip 7pt 

As for $V$ of dimension $d=4$, in order to state our result, we first need to recall that global generation of $E$ gives rise to the exact sequence: 
\begin{equation}\label{eq:mukailaz}
0 \to M_E \to H^0(E) \otimes \mathcal O_V \stackrel{ev}{\longrightarrow} E \to 0,
\end{equation}
where $M_E$ is the so called {\em Lazarsfeld-Mukai bundle} associated to $E$. 
Tensoring with $E$ and passing to cohomology, one has a natural induced map 
\begin{equation}\label{eq:muE}
H^0(E)^{\otimes 2} \stackrel{\mu_E}{\longrightarrow} H^0(E^{\otimes 2}).
\end{equation}

\vskip 7pt

\noindent
{\bf Theorem} (cf. Theorem \ref{thm:4fold} below) {\em Let $V$ be any smooth, irreducible projective four-fold. 
Let $E$ be a globally--generated, rank--$r$ vector bundle on $V$, $r \geqslant 2$, such that $h^0(E) \geqslant r+4$. 
Assume further that: 
\begin{eqnarray*}
q(V) := h^1(\mathcal O_V) & = & 0, \nonumber \\
h^i(V, (\det E)^{-1}) & = &  0, \; 1 \leqslant i \leqslant 3, \\
h^3( V, E^{\vee} \otimes (\det E)^{-1}) & = &  0, \nonumber\\
\mu_E \;\; {\rm is} \;\; {\rm injective}. & &\nonumber
\end{eqnarray*}
Then $E$ is a big vector bundle on $V$.}

\vskip 7pt

In order to write down the cohomological constraints appearing in both theorems, we assume that the $d^{th}$ Segre class of $E$ vanishes; so does the top Chern class of a suitable 
rank-$d$, associated vector bundle $N^{\vee}$ on $V$, where $N$ is the kernel of the evaluation map from $W\otimes {\mathcal O}_V$ to $E$, $W$ being 
a general subspace of $H^0(E)$ of dimension $r+d$. This would imply the existence of a nowhere vanishing section of $N^{\vee}$. The conditions in Theorem \ref{thm:surf} and \ref{thm:4fold} are sufficient to contradict the existence of any such section.

In principle, our results can be extended to any even dimension but the cohomological conditions are in fact more complicated to be written down. Indeed, already in the case of surfaces and four-folds, we need to investigate exterior powers of vector bundles that are defined in terms of short exact sequences. This is possible via successive short exact sequences which become more numerous, as the dimension of $V$ increases. Nonetheless, already in dimension $2$ and $4$ our results give interesting applications. 

As for dimension $2$, Theorem \ref{thm:surf} can be viewed as the "bigness"-version of the ampleness criterion given in \cite[Prop.\;1]{Beau}. Theorem \ref{thm:surf} applies to 
any smooth, projective surface and to any vector bundle $E$ on it, which is globally generated and has arbitrary rank, not only two; moreover, we do not assume the Neron- Severi group to be cyclic and generated by $c_1(E)$. In Section \ref{ExSurf}, we explore some of the various applications of Theorem \ref{thm:surf}; in Example (a) and Example (b), we exhibit vector bundles that are big but not ample. In Example (c), we discuss unsplit vector bundles on Segre-Hirzebruch surfaces ${\mathbb F}_e$, which turn out to be very ample. In Example (d) the reader may find split vector bundles of higher rank. Possible other applications, along the lines of Example (c), give unsplit vector bundles of rank higher than two. \color{black} Example (e), which has been inspired by questions of the Referee, shows that condition 
$h^1((\det E)^{-1}) = 0$ in Theorem \ref{thm:surf} is sufficient, but not necessary, for bigness. \color{black}

As for dimension $4$, we present two possible applications. First, let $V$ be a Fano manifold, i.e. a smooth projective variety such that the anti-canonical is ample. To start with, as proved, for instance, in \cite[Proposition\;4.1]{hsiao}, if $T_V$ is nef and big, then $V$ is a Fano manifold. Conversely, as, for instance, in loc. cit., Question 4.5., one might ask $$\mbox{\it If $V$ is Fano with nef tangent bundle $T_V$, is it true that $T_V$ is big?}.$$ As explained in \cite[p.\,1550098-8]{hsiao}, the affirmative answer to the previous question has been proved up to dimension $3$. Theorem \ref{thm:4fold} allows us to answer this question in dimension $4$ under the assumption $E$ is globally generated and $h^0(E) \geq 9$. Inspired by  
\cite[Prop.\;2]{Beau}, dealing with Lazarsfeld-Mukai bundles on a smooth surface of irregularity $0$ with cyclic Neron-Severi group, we also discuss examples of suitable Lazarsfeld-Mukai bundles on four-folds which turn out to be big but which satisfy all but one of the assumptions in Theorem \ref{thm:4fold} below, proving that Theorem \ref{thm:4fold} gives sufficient but non-necessary conditions for bigness. \color{black}

As for the plan of the paper, in Section \ref{ChernSegre} we recall some preliminary results, in particular on Chern and Segre classes, as well as on positivity on vector bundles. In Section \ref{Surf}, we prove Theorem \ref{thm:surf} and some possible applications of it. Finally, in Section \ref{4fold} we pass to dimension $4$. 

In what follows, we work over the complex field $\mathbb{C}$.  For any smooth, projective variety $V$,
$A_n(V)$ will denote the group of $n$-cycles modulo rational equivalence on $V$, where $0 \leqslant n \leqslant \dim (V)$ (cf. \cite[\S\,1]{Fu}). Unless otherwise stated, 
from now on we will set $d := \dim (V)$ and 
$E$ a vector bundle of rank $r$ on $V$. The {\em dual bundle} of $E$ will be denoted by $E^{\vee}$, unless $E=L$ is a line bundle whose dual will be simply denoted 
by $L^{-1}$. For not reminded terminology and notation, we refer the reader to \cite{H}.


\section{Preliminaries}\label{Pre} We briefly recall some results which are frequently used in the paper.


\subsection{Chern and Segre classes}\label{ChernSegre} For $V$ and $E$ as above, we set $\mathbb{P}(E) := Proj (Sym(E))$ (i.e. $\mathbb{P}(E)$ is the projective--bundle parametrizing $1$-dimensional quotients of the fibres of $E$), 
${\mathcal O}_{\mathbb{P}(E)} (1)$ the {\em tautological line--bundle} on $\mathbb{P}(E)$ and $\mathbb{P}(E) \stackrel{\pi}{\longrightarrow} V$ the canonical projection (cf. e.g. \cite{H}).  By \cite[\S\,1--3]{Fu}, there are homomorphisms
$$A_n(V) \to A_{n-k}(V), \;\; \alpha \to s_k(E) \cap \alpha,$$which are defined by the formula 
\begin{equation}\label{eq:IID4}
s_k(E) \cap \alpha := \pi_*\left(c_1({\mathcal O}_{\mathbb{P}(E)} (1))^{r-1+k} \cap \pi^* (\alpha)\right),
\end{equation}where $\pi^*: A_n(V) \to A_{n+r-1}(\mathbb{P}(E))$ is the flat pull-back (cf.\;\cite[\S,1.7]{Fu}), $\pi_*: 
A_{n-k} ( \mathbb{P}(E)) \to A_{n-k} ( V)$ the push-forward (cf.\;\cite[\S\,1.4]{Fu}) whereas 
$c_1({\mathcal O}_{\mathbb{P}(E)} (1))^{r-1+k}\cap \;_{-}  : A_{r-1 +k} (\mathbb{P}(E)) \to A_{n-k} ( \mathbb{P}(E))$ 
the iterated first Chern class homomorphism (cf.\;\cite[\S\,2.5]{Fu}). 

$s_k(E)$ in \eqref{eq:IID4} is called the $k^{th}$--{\em Segre class of} $E$ whereas $s(E) := 1 + s_1(E) + s_2(E) + \cdots $ the  
{\em total Segre class} of $E$. $s_k(E)$ is a polynomial in the Chern classes $c_1(E), \ldots, c_r(E)$ of $E$; indeed given the Chern polynomial of $E$, 
$$c_E(t) := \sum_{k=0}^r c_k(E) t^k = 1 + c_1(E) t + c_2(E) t^2 + \cdots + c_r(E) t^r,$$the Segre classes defined in \eqref{eq:IID4} turn out to be 
coefficients of the formal power series 
$$s_E (t) := \sum_{k=0}^{+\infty} s_k(E) t^k = 1 + s_1(E) t + s_2(E) t^2 + \cdots$$defined to be the inverse power series of $c_E(t)$, i.e. $s_E(t) = c_E(t)^{-1}$ 
(cf.\;e.g.\;\cite[\S\,3.2]{Fu}). Explicitely, one has (cf.\;also\;\cite[Examples\;8.3.3--8.3.5]{Laz2}):  
\begin{equation}\label{eq:segre}
c_1(E) = - s_1(E), \; c_2(E) = s_1(E)^2 - s_2(E), \ldots, c_k(E) = - s_1(E) c_{k-1}(E) - s_2(E) c_{k-2}(E) - \cdots - s_k(E),\; \forall \; k \geqslant 3.
\end{equation}
If $L$ is any line bundle on $V$, then one has (cf. \cite[Rem.\;3.2.3\;(a),\;Ex.\;3.2.2,\;Ex.\;3.1.1]{Fu}): 
\begin{equation}\label{eq:dualchern}
c_k(E^{\vee}) = (-1)^k c_k(E) \;\; {\rm and} \;\; c_k(E \otimes L) = \sum_{j=0}^k {r-j \choose k - j} c_j(E) c_1(L)^{k-j}, \;\; 1 \leqslant k \leqslant r,
\end{equation}
\begin{equation}\label{eq:dualsegre}
s_k(E^{\vee}) = (-1)^k s_k(E) \;\; {\rm and} \;\; s_k(E \otimes L) = \sum_{j=0}^k (-1)^{k-j}{r-1+k \choose r-1 + j} s_j(E) c_1(L)^{k-j},
\end{equation}where $c_0(E) = s_0(E) =1$ and where $c_1(L)^{k-j}$ denotes the $(k-j)^{th}$ self--intersection of $c_1(L)$.


\subsection{Positivity of vector bundles}\label{Positivity} We remind some definitions concerning certain {\em dimension} and {\em positivity} notions related to vector bundles over a smooth, projective variety $V$ 
from \cite{DPS}, \cite[\S\;II]{GG} and \cite[\S\;6-8]{Laz2}. These concepts will be first reminded for line bundles $L$ on $V$ and then, for vector bundles $E$ of rank $r \geqslant 2$, 
the definitions being related via the canonical association $ E \to V \;\; \leadsto \;\; {\mathcal O}_{\mathbb{P}(E)} (1) \to \mathbb{P}(E)$.

\vskip 5pt

\noindent
$\bullet$ {\bf Kodaira--Iitaka dimension, bigness, nefness}. Take $L$ any line bundle on $V$; its {\em Kodaira--Iitaka dimension}, denoted by $k(L)$, is defined as follows:
\[
k(L) := \left\{\begin{array} {cl} 
- \infty & {\rm if} \; h^0(L^{\otimes m}) = 0, \; \forall \; m \in \mathbb{N} \\
{\rm max}_{m \in \mathbb{N}} \; \dim (\varphi_{L^{\otimes m}}(V)), & {\rm otherwise} 
\end{array}
\right.
\]where $V \stackrel{\varphi_{L^{\otimes m}}}\dashrightarrow \mathbb{P} (H^0(L^{\otimes m})^{\vee})$ denotes the rational map given by the linear system $|L^{\otimes m}|$ (cf. e.g. \cite[\S\;II.A]{GG}). Then, 
\begin{equation}\label{eq:bigL}
L \; \mbox{is said to be {\em big} if} \; k(L) = \dim (V)
\end{equation}(cf. \cite[Def.\;2.2.1]{Laz2}). Finally, $L$ is said to be {\em nef} if $L \cdot C \geqslant 0$ for any effective curve $C \subset V$.

Let now $E$ be any rank--$r$ vector bundle on $V$, with $r \geqslant 2$. Similary as above, its {\em Kodaira--Iitaka dimension} $k(E)$ is defined to be $k(E):= 
k({\mathcal O}_{\mathbb{P}(E)} (1))$. $E$ is said to be a {\em big} vector bundle 
if ${\mathcal O}_{\mathbb{P}(E)} (1)$ is a big line bundle on $\mathbb{P}(E)$ (cf.\;e.g.\cite[Ex.\;6.1.23]{Laz2}).  From \eqref{eq:bigL}, we have therefore 
\begin{equation}\label{eq:bigE}
E \; \mbox{is a big vector bundle if and only if} \; k(E) = \dim (V) + r -1.
\end{equation} Finally, $E$ is said to be {\em nef} if ${\mathcal O}_{\mathbb{P}(E)} (1)$ is a nef line bundle on $\mathbb{P}(E)$ (cf.\;e.g. \cite[Definition\;1.9]{DPS}).

\begin{remark}\label{rem:numerical} Assume that $E$ is globally generated, then $E$ is nef. Indeed, taking ${\mathbb P}(E) \stackrel{\pi}{\rightarrow} V$ the natural projection, global generation of $E$ ensures that $\pi^*E$ is globally generated. 
Since ${\mathcal O}_{{\mathbb P}(E)}(1)$ is a quotient of $\pi^*E$, the tautological line bundle ${\mathcal O}_{{\mathbb P}(E)}(1)$ is globally generated too. Hence, since $|{\mathcal O}_{{\mathbb P}(E)}(1)|$ defines a morphism to a suitable 
projective space $\mathbb{P}$, then ${\mathcal O}_{{\mathbb P}(E)}(1)$ is nef because it is the pull-back via this morphism of the very-ample line bundle ${\mathcal O}_{\mathbb P}(1)$, proving the assertion. 
\end{remark}

\vskip 5pt

\noindent
$\bullet$ {\bf Numerical dimension}. As above, we start with the line bundle case.

\begin{definition} (cf.\cite[II.E,\,p.\,24]{GG}) \label{def:numdim} Let $L$ be any nef line bundle. The {\em numerical dimension} of $L$ is defined to be 
the largest integer $n(L)$ such that $c_1(L)^{n(L)} \neq 0$.
\end{definition}

\noindent
Relating the Kodaira-Iitaka and the numerical dimensions of a nef line bundle $L$, from \cite{Dem} one has (cf.\;also\;\cite[(II.E.1),\;p.24]{GG}): 
\begin{equation}\label{eq:IIE1}
k(L) \leqslant n(L), \; \mbox{and equality holds if} \; n(L) = 0,d.
\end{equation}

Let now $E$ be a globally generated vector bundle, of rank $r \geqslant 2$. From Remark \ref{rem:numerical} $E$ is nef, i.e. ${\mathcal O}_{{\mathbb P}(E)}(1)$ is a nef line bundle on ${\mathbb P}(E)$. Therefore, it makes sense to consider 
the numerical dimension of such a nef line bundle. Indeed, in accordance with \cite[\S\,II.E,\;p.25]{GG}, we set 

\begin{definition}\label{def:numdim2} Let $E$ be a globally generated vector bundle of rank $r$ on $V$. The {\em numerical dimension} of $E$ is 
$n(E) := n({\mathcal O}_{\mathbb{P}(E)} (1))$.
\end{definition}

\noindent
Notice that, since ${\mathcal O}_{\mathbb{P}(E)} (1)$ is very--ample on the fibres of the projection $\mathbb{P}(E) \stackrel{\pi}{\longrightarrow} V$, one has
\begin{equation}\label{eq:boundsE}
r-1 \leqslant n(E) \leqslant \dim(\mathbb{P}(E)) = \dim(V) + r - 1 = d + r -1.
\end{equation}

On the other hand, since $E$ is nef, by \eqref{eq:IID4} and Definition \ref{def:numdim} we have 
\begin{equation}\label{eq:IIE2}
n(E) \; \mbox{is the largest integer with} \; s_{n(E)-r+1}(E) \neq 0.  
\end{equation} 

\noindent
Notice that \eqref{eq:IIE2} coincides with (II.E.2) in \cite{GG}, where the authors consider a wider class of vector bundles.

Taking into account the definition of Kodaira-Iitaka dimension $k(E)$ above, \eqref{eq:IIE1} and Definition \ref{def:numdim2}, one has therefore
\begin{equation}\label{eq:IIG8}
k(E) \leqslant n(E), \; \mbox{where the equality holds when} \; n(E) = \dim(V) + r-1 = d+ r-1. 
\end{equation}

Notice that $n(E) = d+r-1$ (and so $k(E) = n(E) = \dim (\mathbb{P}(E))$) implies that $E$ is big. Moreover, by the global generation (and so nefness) of $E$, $s_{n(E)-r+1}(E) = s_d (E) \neq 0$ is equivalent to $s_d(E) > 0$, as it follows from \cite[Theorem\;2.5]{DPS}, with $d = k$ and $Y = V$, which applies to the Segre class $s_d(E)$ considered as a suitable Schur polynomial 
(cf. the "second interesting example" after \cite[Theorem\;2.5]{DPS}).

To sum up, for a globally generated rank-$r$ vector bundle $E$ on a $d$-dimensional smooth projective variety $V$, the 
bigness of $E$ is encoded by the positivity of the $d^{th}$ Segre class of $E$. In the sequel, we will be concerned in finding sufficient cohomological conditions on a globally generated vector bundle $E$ 
ensuring the positivity of $s_d(E)$. 


\section{The surface case}\label{Surf} In this section, $d=2$. Inspired by \cite[Lemma]{Beau}, one can prove the following

\begin{theorem}\label{thm:surf} Let $V$ be any smooth, irreducible projective surface. 
Let $E$ be a globally--generated, rank--$r$ vector bundle on $V$, $r \geqslant 2$, such that $h^0(E) \geqslant r+2$. 
Assume further that $h^1((\det E)^{-1}) = 0$. Then $E$ is a big vector bundle on $V$.
\end{theorem}

\begin{proof} Since $E$ is globally generated, one has 
\begin{equation}\label{eq:Laz615}
H^0(E) \otimes \mathcal O_V \stackrel{ev}{\longrightarrow} E \to 0. 
\end{equation} When $h^0(E) = r+2$, we set $W:= H^0(E)$. When otherwise  $h^0(E) > r+2$, we take $W \subset H^0(E)$ corresponding to the general point of the Grassmannian 
$\mathbb{G}(r+2, H^0(E))$ parametrizing $(r+2)$-dimensional sub-vector spaces of $H^0(E)$. As in 
\cite[Ex.\;6.1.5,\;p.9]{Laz2}, \eqref{eq:Laz615} defines a morphism $$\mathbb{P}(E) \stackrel{\phi}{\longrightarrow} \mathbb{P}(H^0(E)) = Proj(Sym(H^0(E)))$$(i.e. 
the projective space of one-dimensional quotients of $H^0(E)$, equivalently of one-dimensional sub-vector spaces of $H^0(E)^{\vee}$). 
Then $x := \dim ({\rm Im} (\phi)) \leqslant \dim (\mathbb{P}(E)) = r+1$. From the definition of $W$, 
the surjection $H^0(E)^{\vee} \to W^{\vee} \to 0$ gives rise to the linear projection 
$$ \mathbb{P}(H^0(E)) \stackrel{\pi_{\Lambda}}{\dashrightarrow}  \mathbb{P}(W),$$whose center $\Lambda$ is a linear subspace of $\mathbb{P}(H^0(E))$ of 
dimension $h^0(E) - r -3$. The generality of $W$ implies that $\Lambda$ is a general linear subspace of ${\mathbb P}(H^0(E))$. Thus, since $x \leq r+1$, the subvariety $\Lambda \cap Im(\phi)$ is empty, which implies that 
$\pi_{\Lambda} \circ \phi: \mathbb{P}(E) \to \mathbb{P}(W)$ is a morphism. 

To sum up, in any case one has the exact sequence: 
\begin{equation}\label{eq:N}
0 \to N \to W \otimes \mathcal O_V \stackrel{ev_W}{\longrightarrow} E \to 0,
\end{equation}where $N := {\rm ker} (ev_W)$ is a rank--$2$ vector bundle on $V$. Dualizing \eqref{eq:N} shows that $N^{\vee}$ is globally generated. 
Let $\sigma \in H^0(N^{\vee})$ be a general section; then the zero--locus $V(\sigma) \subset V$ is a zero--dimensional scheme of length 
$c_2(N^{\vee}) \geqslant 0$. From \eqref{eq:dualchern}, one has also $c_2(N) \geqslant 0$. 

By the exact sequence \eqref{eq:N}, the total Chern classes of $E$ and $N$ satisfy $c(E) c(N) = 1$, thus $c(N) = s(E)$, where 
$s(E)$ \color{black} is \color{black} the total Segre class of $E$ as in \S\;\ref{ChernSegre}. From \eqref{eq:segre}, one gets therefore
$ 0 \leqslant c_2(N) = s_2(E) = c_1(E)^2 - c_2(E)$.

If $ 0 < s_2(E)$, from \eqref{eq:IIE2} and the nefness of $E$ (cf. Rem.\;\ref{rem:numerical}), it follows that $n(E) - r + 1 \geqslant 2$, where $n(E)$ is as in Definition\;\ref{def:numdim2}. In such a case one has $n(E) \geqslant r+1$. By \eqref{eq:boundsE}, one therefore concludes that $n(E) = r+1 = \dim(V) + (r-1)$ which, by \eqref{eq:IIG8}, implies 
$n(E) = k(E) = r+1 = \dim (\mathbb{P}(E))$. This gives that ${\mathcal O}_{\mathbb{P}(E)} (1)$ is a big line bundle, as it follows from \eqref{eq:bigL}, so 
that $E$ is a big vector bundle. 

We want to show that, under our assumptions, the case $s_2(E) = 0$ cannot occur. \color{black} To do this, we use the same argument as in \cite[Proof of the Lemma]{Beau}. \color{black} Assume by contradiction that $s_2(E) = 0$, so also  
$c_2(N) = c_2(N^{\vee})  = 0$. This implies that $\sigma \in H^0(N^{\vee})$ general as above is no--where vanishing on $V$ giving rise to 
the exact sequence $$ 0 \to \mathcal O_V \stackrel{\cdot\;\sigma}{\longrightarrow} N^{\vee} \to \det\;N^{\vee} \cong \det\;E\to 0,$$the isomorphism on the right--side following from 
\eqref{eq:N}. The previous exact sequence shows that 
$$N^{\vee} \in {\rm Ext}^1(\det\;E, \mathcal O_V) \cong H^1((\det\;E)^{-1}) = (0),$$the latter equality following from assumptions. Therefore 
$N^{\vee} = \mathcal O_V \oplus \det\;E$, i.e. $N = \mathcal O_V \oplus (\det\;E)^{-1}$. Plugging into \eqref{eq:N} gives
$$0 \to \mathcal O_V \oplus (\det\;E)^{-1} \to W \otimes \mathcal O_V \cong \mathcal O_V ^{\oplus (r+2)} \to E \to 0$$from which one deduces 
the exact sequence 
$$0 \to (\det\;E)^{-1} \to \mathcal O_V ^{\oplus (r+1)} \to E \to 0.$$Since $h^1((\det\;E)^{-1}) = 0$, the previous exact sequence implies 
$h^0(E) \leqslant r+1$, which contradicts assumptions. 
\end{proof}

\vskip 5pt


\subsection{Examples}\label{ExSurf} We discuss some examples which satisfy assumptions in Theorem \ref{thm:surf}. 

\vskip 7pt

\noindent
$(a)$ Let $V = \mathbb{P}^2$ and consider the rank-$2$ vector bundle $E:= \mathcal O_{\mathbb{P}^2} \oplus \mathcal O_{\mathbb{P}^2}(2)$. 
The vector bundle $E$ is globally generated, with $h^0(E) = 7$ and $h^1((\det E)^{-1}) = h^1(\mathcal O_{\mathbb{P}^2}(-2)) =0$. From Theorem \ref{thm:surf}, $E$ is big. Indeed, 
$|\mathcal O_{\mathbb{P}(E)}(1)|$ maps $\mathbb{P}(E)$ in $\mathbb{P}^6$ onto the cone over the Veronese surface in $\mathbb P^5$; in particular, $E$ is big but not ample. \color{black} Another example
in the same vein is e.g. $E := \mathcal O_{\mathbb{P}^2} \oplus T_{\mathbb{P}^2}$, \color{black} where \color{black} $T_{\mathbb{P}^2}$ \color{black} is \color{black} the tangent bundle on ${\mathbb P}^2$, as it follows from the Euler sequence for $T_{\mathbb{P}^2}$ and $c_1(T_{\mathbb{P}^2}) = {\mathcal O}_{\mathbb{P}^2} (3)$.\color{black}

\vskip 7pt

\noindent
$(b)$ The previous example can be easily extended to any smooth, projective irreducible surface $V$ and any rank-$r$ vector bundle on $V$ of the form $E = \mathcal O_V \oplus F$, with $F$ any ample, rank-$(r-1)$ vector bundle such that $h^0(F) \geqslant r+1$ and $h^1((\det F)^{-1}) = 0$. For example consider $\FF_e:= \Pp(\Oc_{\Pp^1} \oplus\Oc_{\Pp^1}(-e))$ the Hirzebruch surface, for some integer $e \geqslant 0$. We let $\pi_e : \FF_e \to \Pp^1$ denote the natural projection. Thus 
${\rm Num}(\FF_e) = \mathbb{Z}[C_e] \oplus \mathbb{Z}[f]$, where $C_e$  is a section of $\FF_e$ corresponding to the surjection 
$\Oc_{\Pp^1} \oplus\Oc_{\Pp^1}(-e) \to\!\!\!\to \Oc_{\Pp^1}(-e)$ on $\Pp^1$ ($C_e$ is unique when $e >0$\color{black}), and $f = \pi^*(p)$, for any $p \in \Pp^1$, the class of a fibre;  in particular 
$C_e^2 = - e, \; f^2 = 0, \; C_ef = 1$. Let $b$ be an integer and assume $b > e \geqslant  0 \color{black}$; consider the vector bundle $E = \mathcal O_{\FF_e} \oplus \mathcal O_{\FF_e}(C_e + b f)$. 
Since $b > e$, by \cite[V.Thm.\;2.17,(b)]{H}, $\mathcal O_{\FF_e}(C_e + b f)$ is very--ample; thus $E$ is globally generated but not 
ample. Moreover $$h^0(E) = 1 + h^0(\FF_e, \mathcal O_{\FF_e}(C_e + b f)) = 1 + h^0( \Pp^1, \left(\Oc_{\Pp^1} \oplus\Oc_{\Pp^1}(-e)\right) \otimes \Oc_{\Pp^1}(b))$$the second equality following from Leray isomorphism. Since $b>e$, then 
$$h^0( \Pp^1, \left(\Oc_{\Pp^1} \oplus\Oc_{\Pp^1}(-e)\right) \otimes \Oc_{\Pp^1}(b)) = h^0 (\Pp^1, \left(\Oc_{\Pp^1} (b) \oplus\Oc_{\Pp^1}(b-e)\right) = 2b+2-e.$$Thus 
$h^0(E) = 2b+3-e >4$, as it follows by $b> e \geqslant  0 \color{black}$. Finally 
$$h^1((\det E)^{-1}) = h^1(\mathcal O_{\FF_e}(-C_e - b f)) = h^1(\omega_{\FF_e} \otimes \mathcal O_{\FF_e}(C_e + b f)) = 0$$where 
the second equality follows from Serre duality whereas the last equality from Kodaira vanishing. By Theorem \ref{thm:surf}, it follows that $E$ is big.

\vskip 7pt

\noindent
$(c)$ To discuss  examples of big (or even very-ample)  unsplitting vector bundles, we use same notation as in Example $(b)$ above and consider integers  $e \geqslant 1$, $b \geqslant 4e + 3$ and 
$\frac{3b+2-4e}{2} \leqslant k < 2b -4e$. Let $A := \mathcal O_{\FF_e} (2 C_e + (2b-k-2e) f)$ and $B := \mathcal O_{\FF_e} (C_e + (k - b + 2e) f)$ be line bundles on $\mathbb{F}_e$. Any $u \in {\rm Ext}^1(B, A)$ 
gives rise to a rank-two vector bundle $E_u$ fitting in the exact sequence 
\begin{equation}\label{eq:al-be}
0 \to A \to E_u \to B \to 0,
\end{equation}with $\det(E_u) = \mathcal O_{\FF_e} (3 C_e + b f)$. Notice that 
$$\dim({\rm Ext}^1(B,A)) = 9e+4k-6b-2 \geqslant e+2 \geqslant 3,$$where the inequalities follow \color{black} from numerical assumptions $k \geqslant \frac{3b+2-4e}{2}$ and $e \geqslant 1$. 
To show the equality, consider 
{\small $$\dim {\rm Ext}^1(B,A) = h^1(\FF_e, A \otimes B^{-1})= h^1(\FF_e, \mathcal O_{\FF_e} ( C_e + (3b-2k-4e)f)) = 
h^1(\Pp^1, (\Oc_{\Pp^1} \oplus \Oc_{\Pp^1} (-e)) \otimes \Oc_{\Pp^1} (3b-2k-4e)).$$}Since $k \geqslant \frac{3b+2-4e}{2}$ 
both $h^1( \Oc_{\Pp^1} (3b-2k-4e))$ and $h^1( \Oc_{\Pp^1} (3b-2k-5e))$ are positive and they add--up to $9e+4k-6b-2$. 

We claim that the general $u \in {\rm Ext}^1(B,A)$ gives rise to an unsplitting vector bundle. To prove this, we use that $E_u$ is of rank-two and that it fits 
in the exact sequence \eqref{eq:al-be}, thus $E_u^{\vee} \cong E_u \otimes A^{\vee} \otimes B^{\vee}$, since $\det E_u = A \otimes B$. 
Tensoring \eqref{eq:al-be} respectively by $E_u^{\vee}\cong E_u \otimes A^{\vee} \otimes B^{\vee}$,  $B^{\vee}$, $A^{\vee}$, we get the following exact diagram
\begin{equation}\label{eq:diag}
\begin{array}{rcccccc}
      &  0 &     &    0        &     & 0 & \\

      & \downarrow  &     &     \downarrow       &     & \downarrow & \\

 0 \to   &  A \otimes B^{\vee} &    \to  & E_u \otimes B^{\vee} &   \to  & \Oc_{\FF_e} & \to 0 \\

      &  \downarrow &     &       \downarrow     &     & \downarrow & \\

 \;\;\;0 \to &  E_u \otimes B^{\vee} & \to & E_u \otimes E_u^{\vee} & \to & E_u \otimes A^{\vee} & \to 0 \\

     &   \downarrow &     &       \downarrow     &     & \downarrow & \\

    0 \to  &  \Oc_{\FF_e} &      \to &     E_u \otimes A^{\vee} & \longrightarrow   & B \otimes A^{\vee} & \to 0 \\

          &     \downarrow &     &     \downarrow       &     & \downarrow & \\

          &  0 &     &           0 &     & 0 & .
\end{array}
\end{equation} One needs to compute $h^0(E_u \otimes B^{\vee})$ and $h^0(E_u \otimes A^{\vee})$. 
From the cohomology sequence associated to the first row of diagram \eqref{eq:diag} we get
$$0\to H^0(A \otimes B^{\vee} )\to H^0(E_u \otimes B^{\vee})\to H^0(\Oc_{\FF_e}) \stackrel{\widehat{\partial}}{\longrightarrow} H^1(A \otimes B^{\vee}).$$The coboundary map 
$$H^0(\Oc_{\FF_e}) \stackrel{\widehat{\partial}}{\longrightarrow} H^1(A \otimes B^{\vee}) \cong {\rm Ext}^1(B,A),$$ has to be injective 
since it corresponds to the choice of the non-trivial general extension class
$u \in {\rm Ext}^1(B,A)$ associated to $E_u$. Thus one gets
\begin{equation}\label{hOalpha}
h^0(E_u \otimes B^{\vee} )=h^0(A \otimes B^{\vee})= 
h^0(\Oc_{\Pp^1} (3b-2k-4e)) + h^0(\Oc_{\Pp^1} (3b-2k-5e)) = 0, 
\end{equation} the last equality following from the assumption $k \geqslant \frac{3b+2-4e}{2}$, which gives $3b-2k-4e \leqslant -2$ and $3b-2k-5e \leqslant -2-e$. 

From the third row of diagram \eqref{eq:diag}, since $B \otimes A^{\vee} = \mathcal O_{\FF_e} (- C_e + (2k - 3b + 4e) f)$ is not effective, it follows that 
$h^0( E_u \otimes A^{\vee})=h^0(\Oc_{\FF_e})=1$, thus $H^0( E_u \otimes A^{\vee})\cong {\mathbb C}$.

From the second column of diagram \eqref{eq:diag}, we have
$$0\to  H^0(E_u \otimes B^{\vee})\to H^0( E_u \otimes E_u^{\vee})\stackrel{\psi}{\longrightarrow} H^0(E_u \otimes A^{\vee}) \cong \mathbb{C}  \to H^1(E_u \otimes B^{\vee})\to \cdots .$$We claim that the map $\psi$ is surjective. To prove this, notice that from the first two columns of diagram \eqref{eq:diag} and the fact that the coboundary map ${\widehat{\partial} }$ is injective 
(as remarked above) we have
\[
\begin{array}{rcccccc}
   & &  0  &  &    & H^0(E_u \otimes E_u^{\vee}) & \\

      &   & \downarrow    &       &     & \downarrow^{\psi}  & \\

& 0 \to &    H^0(\Oc_{\FF_e})\ & \stackrel{\cong}{\longrightarrow}  & &H^0( E_u \otimes A^{\vee}) & \to 0 \\

     &    &            \downarrow^{\widehat{\partial}}     &  &   & \downarrow^{\tilde{\partial} }& \\

     &  &           H^1(A \otimes B^{\vee}) & \longrightarrow &   &H^1( E_u \otimes B^{\vee}) &  . \\
\end{array}
\]Since $H^0(E_u \otimes A^{\vee}) \cong \mathbb{C}$, $\psi$ is not surjective if and only if  $\psi = 0$, which is equivalent 
to ${\tilde{\partial} }$ to be injective. The latter is impossible since, from the first column of diagram \eqref{eq:diag},  we have
$$H^0(\Oc_{\FF_e})  \stackrel{\widehat{\partial}}{\longrightarrow} H^1(A \otimes B^{\vee})  \to H^1( E_u \otimes B^{\vee})$$
and the composition of the above two maps is ${\tilde{\partial} }$. From the surjectivity of $\psi$,  we conclude that
\begin{equation}\label{endom}
h^0( E_u \otimes E_u^{\vee})=h^0(E_u \otimes B^{\vee})+1. 
\end{equation}Combining \eqref{hOalpha} and \eqref{endom} we determine 
$h^0( E_u \otimes E_u^{\vee}) =1$ when $u \in {\rm Ext}^1(B,A)$  is \color{black} general. Since $E_u$ is simple, we deduce that $E_u$ must be unsplitting.

Once we have produced $E_u$ unsplitting  with arguments above, we want to show that it satisfies assumptions as in Theorem \ref{thm:surf}, so $E_u$ will be big. \color{black} 
Notice indeed \color{black} that $$h^0(E_u) \geqslant 4b - 6 e - k + 5> 6e+11  > rk (E_u) + 2; \color{black}$$the second inequality above follows 
from the assumptions $k< 2b - 4e$ and $b \geqslant 4e +3$, whereas the first inequality is a consequence of: (1) standard computations 
which show that $h^2(A) = h^j(B) = 0$, for $j \geqslant 1$,  so $h^2(E_u) = 0$;  (2) the exact sequence \eqref{eq:al-be}, from which one gets 
$$h^0(E_u) = h^0(A) + h^0(B) - h^1(A) + h^1(E_u) = \chi(A) + \chi(B) + h^1(E_u) \geqslant \chi(A) + \chi(B)$$where 
the latter can be easily computed via Riemann-Roch (left to the reader).   Furthermore, \color{black} 
$$h^1((\det E_u)^{\vee}) = h^1(\mathcal O_{\FF_e} (-3 C_e - b f)) = h^1(\omega_{\FF_e} \otimes \mathcal O_{\FF_e} (3 C_e + b f)) = 0$$where the second equality follows from Serre duality whereas the third from the Kodaira vanishing theorem and the very-ampleness of $\mathcal O_{\FF_e} (3 C_e + b f)$, as it follows from \cite[V.Cor.2.18\;(a)]{H} and from $b \geqslant 4e+3$. Since all the assumptions of Theorem \ref{thm:surf} are satisfied, \color{black} it follows that $E_u$ is a big vector bundle.

 One can even show more, namely that $E_u$ is actually \color{black} a very-ample vector bundle. 
Indeed, from \cite[V.Cor.2.18\;(a)]{H} and the assumption $k<2b-4e$, $A$ is a very-ample line bundle. Similarly, since $k\geq \frac{3b+2-4e}{2}$ and $b \geqslant 4e+3$, from \cite[V.Thm.2.17\;(c)]{H} it follows that $B$ is also very-ample. Thus $A \oplus B$ (which corresponds to the zero-vector in ${\rm Ext}^1(B,A)$) is a very-ample vector bundle of rank $2$. 
Since very-ampleness is an open condition, one deduces that $E_u$ is very-ample, when $u \in {\rm Ext}^1(B,A)$ is general. \color{black}

\vskip 5pt

\noindent
$(d)$ Examples of big \color{black} vector bundles with rank higher than two can be easily constructed as follows. Using same notation and assumptions as in (c) above, 
let $E_u$ be as in \eqref{eq:al-be}, corresponding to the general $u \in {\rm Ext}^1(B,A)$. From Example (c), $E_u$ satisfies assumptions of Theorem \ref{thm:surf} and it is also very-ample. 
For any integer $r \geqslant 3$, the vector bundle $\mathcal E_u := \mathcal O_V^{r-2} \oplus E_u$ is of rank 
$r \geqslant 3$, it is globally generated, with $h^0( \mathcal E_u) \geqslant r + e + 9$ and  
$h^1((\det \mathcal E_u)^{-1}) = h^1((\det E_u)^{-1}) =0$ (as it follows from computations in Example (c)). Then, $\mathcal E_u $ is big but not ample. 

Further examples of big (resp., ample or very-ample) vector bundles of rank higher than two can be easily constructed by iterating extension procedure as in Example (c), 
starting from $E_u$ as in (c) and its extension via a globally generated and big (resp., ample or very-ample) line bundle $L$.  

\vskip 5pt

\noindent
$(e)$ We include here an example which shows that condition $h^1((\det E)^{-1}) =0$ in 
Theorem \ref{thm:surf} is actually sufficient but not necessary for bigness.   In other words, we show an example of a smooth projective variety $V$ and a big vector bundle $E$ on it such that  $h^1((\det E)^{-1})  \neq 0$. More specifically, let $V$ be a smooth projective surface that is irregular, i.e., $q(V) \neq 0$. Fix an ample line bundle $A$ on $V$ and consider the rank two vector bundle on $V$ which is defined as follows, namely $U:= {\mathcal O}_V \oplus \left(A^{-1} \otimes A^{-1}\right)$. By definition, $U$ has global sections because $H^0(V,U) =H^0(V, {\mathcal O}_V) \neq 0$. Next, define the rank two vector bundle $E: = U \otimes A$. As explained in \cite[Example 6.1.23, p. 18]{Laz2}, the vector bundle $E$ is big (with notation as therein in this case $m=1$, i.e. the first symmetric power 
$S^m(U)$ for which one has effectivity is for $S^1 U = U$, and $L =A$ is ample).  Let us compute the determinant $\det(E)= \det(U \otimes A)$, which turns out to equal $\det(U) \otimes A^{\otimes 2} \cong \mathcal O_V$. Therefore 
$H^1(V, \det(E)^{-1})= H^1(V, {\mathcal O}_V) \neq \{0\}$;  indeed the dimension of $H^1(V, {\mathcal O}_V)$ is given by $q(V)$, which is different from $0$ by our assumptions on $V$. 
\color{black}


\section{The four--fold case}\label{4fold} Here we focus on the case $d=4$, determining sufficient conditions for bigness of rank $r\geqslant 2$ vector bundles on a 
four-fold $V$. Preliminarly, consider the following general fact; let $V$ be any smooth, projective variety and let $E$ be a globally--generated, rank--$r$ vector bundle on 
$V$, $r \geqslant 2$. Recall that global generation of $E$ gives rise to the exact sequence \eqref{eq:mukailaz}; tensoring with $E$ and passing to cohomology, 
one has the natural induced map in \eqref{eq:muE} and it is straightforward to observe that 
\begin{equation}\label{eq:injectiveaccazero}
\mu_E \; \mbox{is injective} \; \Leftrightarrow \; h^0(M_E \otimes E) =0.
\end{equation}

With this set-up, we prove the following:

\begin{theorem}\label{thm:4fold} Let $V$ be any smooth, irreducible projective four-fold. 
Let $E$ be a globally--generated, rank--$r$ vector bundle on $V$, $r \geqslant 2$, such that $h^0(E) \geqslant r+4$. 
Assume further that: 
\begin{eqnarray}
\label{eq:4fold}
q(V) := h^1(\mathcal O_V) & = & 0, \nonumber \\
h^i(V, (\det E)^{-1}) & = &  0, \; 1 \leqslant i \leqslant 3, \\
h^3( V, E^{\vee} \otimes (\det E)^{-1}) & = &  0, \nonumber\\
\mu_E \;\; {\rm is} \;\; {\rm injective}. & &\nonumber
\end{eqnarray}
Then $E$ is a big vector bundle on $V$.
\end{theorem}

\begin{proof} Reasoning similarly as in the proof of Theorem \ref{thm:surf}, 
for general $W \subseteq H^0(E)$ of dimension $r+4$, one has the exact sequence 
\begin{equation}\label{eq:N4fold}
0 \to N \to W \otimes \mathcal O_V \stackrel{ev_W}{\longrightarrow} E \to 0,
\end{equation}where $N$ is a vector bundle of rank $4$.  Dualizing  \eqref{eq:N4fold} shows that $N^{\vee}$ is globally generated.
 
Let $\sigma \in H^0(N^{\vee})$ be a general section; then the zero--locus $V(\sigma) \subset V$ is a zero--dimensional scheme of length 
$ 0 \leqslant c_4(N^{\vee}) = c_4(N)$,  the equality following from \eqref{eq:dualchern}. Thus, from \eqref{eq:segre}, one gets 
$0 \leqslant c_4(N) = s_4(E)$, where $s_4(E)$ is \color{black} as in \eqref{eq:IID4}--\eqref{eq:segre}.

If $ 0 < s_4(E)$, \eqref{eq:IIE2} and the nefness of $E$ (cf. Rem.\;\ref{rem:numerical}) give $n(E) - r + 1 \geqslant 4$, i.e $n(E) \geqslant r+3$. In such a case, from \eqref{eq:boundsE} one concludes that $n(E) = r+3 = \dim(V) + (r-1)$, which implies that $E$ is a big vector bundle. 

One is therefore left to show that, under assumptions \eqref{eq:4fold}, the case $s_4(E) = 0$ cannot occur. Assume by contradiction that $s_4(E) = 0$, so $c_4(N) = c_4(N^{\vee}) = 0$. 
This implies that $\sigma \in H^0(N^{\vee})$ general as above is no--where vanishing on $V$, giving rise to the exact sequence
\begin{equation}\label{eq:sigma}
0 \to \mathcal O_V \stackrel{\cdot\;\sigma}{\longrightarrow} N^{\vee} \to F \to 0,
\end{equation}where $F$ is a rank-$3$ vector bundle. Dualizing \eqref{eq:sigma}, one gets
\begin{equation}\label{eq:i}
0 \to F^{\vee}\to N \to \mathcal O_V \to 0, 
\end{equation}i.e. $N \in {\rm Ext}^1(\mathcal O_V, F^{\vee}) \cong H^1(F^{\vee})$. If we show that $h^1(F^{\vee}) = 0$, then 
$N = \mathcal O_V \oplus F^{\vee}$ which, plugged into \eqref{eq:N}, gives 
$$0 \to \mathcal O_V \oplus F^{\vee} \to W \otimes \mathcal O_V \cong \mathcal O_V ^{\oplus (r+4)} \to E \to 0$$from which one deduces 
$$0 \to F^{\vee} \to \mathcal O_V ^{\oplus (r+3)} \to E \to 0.$$Condition $h^1(F^{\vee}) = 0$ would therefore imply $h^0(E) \leqslant r+3$, contradicting the assumptions.
 
The rest of the proof is therefore concerned to showing that conditions in \eqref{eq:4fold} guarantee $h^1(F^{\vee}) = 0$. To do this, consider   
\begin{equation}\label{eq:notegil11} 
0 \to \bigwedge^2 F^{\vee} \stackrel{\alpha_2}{\longrightarrow} \bigwedge^2 N \stackrel{\beta_2}{\longrightarrow} F^{\vee} \to 0,
\end{equation} deduced from \eqref{eq:i} and \cite[II.5,\;Ex.\;5.16(d),\;p.127]{H}. Then \eqref{eq:notegil11} gives:   
\begin{equation}\label{eq:condequiv} 
h^1(F^{\vee}) =0 \; \Leftrightarrow \; \left\{\begin{array}{ll}
H^1(\bigwedge^2 F^{\vee}) \stackrel{H^1(\alpha_2)}{\longrightarrow} H^1(\bigwedge^2 N) & {\rm surjective, \; and} \\
H^2(\bigwedge^2 F^{\vee}) \stackrel{H^2(\alpha_2)}{\longrightarrow} H^2(\bigwedge^2 N) & {\rm injective}. 
\end{array}\right.
\end{equation}

We first show the injectivity of the map $H^2(\alpha_2)$. Since $F^{\vee}$ is of rank $3$, from \cite[II.5,\;Ex.\;5.16(d),\;p.127]{H}, one has 
$$\bigwedge^2 F^{\vee} \cong F \otimes (\det F^{\vee}) = F \otimes (\det F)^{-1}.$$ Moreover \eqref{eq:i} gives 
$(\det F)^{-1}  \cong \det N$ whereas \eqref{eq:N4fold} gives $\det N \cong (\det E)^{-1}$, i.e. $\det F \cong \det E$. Thus, the previous isomorphism reads \color{black}
$\bigwedge^2 F^{\vee} \cong F \otimes (\det E)^{-1}$, 
so the map $H^2(\alpha_2)$ reads  
$H^2(F \otimes (\det E)^{-1}) \stackrel{H^2(\alpha_2)}{\longrightarrow} H^2(\bigwedge^2 N)$. Tensoring \eqref{eq:sigma} by $(\det E)^{-1}$ gives: 
\begin{equation}\label{eq:notegil15}
0 \to (\det E)^{-1} \to N^{\vee} \otimes (\det E)^{-1} \to F \otimes (\det E)^{-1} \to 0.
\end{equation}Since, from \eqref{eq:4fold} we have $h^i((\det E)^{-1}) = 0$ for $1 \leqslant i \leqslant 3$,   
\eqref{eq:notegil15} gives $$H^i(N^{\vee} \otimes (\det E)^{-1}) \cong H^i(F \otimes (\det E)^{-1}), \;\; 1 \leqslant i \leqslant 2.$$Dualizing \eqref{eq:N4fold} and tensoring 
with $(\det E)^{-1}$ gives: 
$$0 \to E^{\vee} \otimes (\det E)^{-1} \to W^{\vee} \otimes  (\det E)^{-1} \to N^{\vee} \otimes (\det E)^{-1}\to 0.$$

Since 
$h^2( (\det E)^{-1}) = h^3(E^{\vee} \otimes (\det E)^{-1}) =0$ from  \eqref{eq:4fold}, the previous exact sequence gives 
\linebreak $h^2(N^{\vee} \otimes (\det E)^{-1}) = 0$ which from the isomorphism above implies $h^2(F \otimes (\det E)^{-1}) =0$, proving the injectivity of 
$H^2(\alpha_2)$.  

Concerning the surjectivity of $H^1(\alpha_2)$, consider the exact sequence \eqref{eq:N4fold}. From \cite[II.5,\;Ex.\;5.16(d),\;p.127]{H}, \eqref{eq:N4fold} gives rise to 
a filtration $$0 \subset G^2 \subset G^1 \subset G^0 = \bigwedge^2 W \otimes \mathcal O_V \cong \mathcal O_V^{\oplus {{r+4} \choose {2}}},$$where 
$$G^2 \cong \bigwedge^2 N, \; G^1/G^2 \cong N \otimes E, \; G^0/G^1 \cong \bigwedge^2 E.$$In other words, from \eqref{eq:N4fold} one deduces 
the following exact sequences:  
\begin{equation}\label{eq:notegil23}
0 \to \bigwedge^2 N \to G^1 \to N \otimes E \to 0 \;\;\;\; {\rm and} \;\;\;\; 0 \to G^1 \to \mathcal \bigwedge^2 W \otimes \mathcal O_V \to \bigwedge^2 E \to 0.
\end{equation}

Passing to cohomology in the second exact sequence in \eqref{eq:notegil23} and using assumption $q(V) = h^1(\mathcal O_V) =0$, we get 
\begin{equation}\label{eq:Flam2mag}
0 \to H^0(G^1) \to \bigwedge^2 W \stackrel{{\lambda_E}_|}{\longrightarrow} H^0(\bigwedge^2 E) \stackrel{\pi}{\to} H^1(G^1) \to 0.
\end{equation}

\begin{claim}\label{cl:Flam2mag} The map ${\lambda_E}_|$ in \eqref{eq:Flam2mag} is injective. 
In particular, one has $$H^0(G^1) = 0\;\;\; {\rm and} \;\;\; H^1(G^1) \cong \frac{H^0(\bigwedge^2 E)}{\bigwedge^2 W}.$$
\end{claim}

\begin{proof}[Proof of Claim \ref{cl:Flam2mag}] Consider the map $\mu_E : H^0(E)^{\otimes 2} \to H^0(E^{\otimes 2})$ as in \eqref{eq:muE}. 
On the one hand, one has 
$$H^0(E)^{\otimes 2} = \bigwedge^2 H^0(E) \oplus Sym^2(H^0(E))$$and the map $\mu_E$ then splits as 
$\mu_E = \lambda_E \oplus \sigma_E$, where 
$$\lambda_E := {\mu_E}_{|_{\bigwedge^2 H^0(E)}}: \bigwedge^2 H^0(E) \to H^0(E^{\otimes 2}) \;\;\; {\rm and} \;\;\; \sigma_E := {\mu_E}_{|_{Sym^2(H^0(E))}}: Sym^2(H^0(E)) \to H^0(E^{\otimes 2}).$$On the other hand, 
since $E^{\otimes 2} = \bigwedge^2 E \oplus Sym^2(E)$, then 
$$H^0(E^{\otimes 2}) = H^0(\bigwedge^2 E) \oplus H^0(Sym^2(E));$$therefore, more precisely one has 
$$\lambda_E : \bigwedge^2 H^0(E) \to H^0(\bigwedge^2 E) \;\;\; {\rm and} \;\;\;\sigma_E : Sym^2(H^0(E)) \to H^0(Sym^2(E)).$$

By assumption \eqref{eq:4fold}, the map $\mu_E$ is injective, so $\lambda_E$ is also injective. Since $W \subseteq H^0(E)$, then $\bigwedge^2 W \subseteq \bigwedge^2 H^0(E)$ and the map 
${\lambda_E}_|$ in \eqref{eq:Flam2mag} is nothing but the restriction of $\lambda_E$ to $\bigwedge^2 W$, proving the injectivity of ${\lambda_E}_|$. The rest of the claim easily follows from 
\eqref{eq:Flam2mag}. 
\end{proof}

From the first exact sequence in \eqref{eq:notegil23}, one gets
{\small
\begin{equation}\label{eq:help}
 0 \to H^0(\bigwedge^2 N) \stackrel{\gamma_1}{\longrightarrow} H^0(G^1) \stackrel{\gamma_2}{\longrightarrow} H^0(N \otimes E) \stackrel{\gamma_3}{\longrightarrow} H^1(\bigwedge^2 N) 
\stackrel{\gamma_4}{\longrightarrow} H^1(G^1) \stackrel{\gamma_5}{\longrightarrow} H^1(N\otimes E) \to \cdots.
\end{equation}}From Claim \ref{cl:Flam2mag}, \eqref{eq:help} reduces to 
$$ 0 \to H^0(N \otimes E) \stackrel{\gamma_3}{\longrightarrow} H^1(\bigwedge^2 N) 
\stackrel{\gamma_4}{\longrightarrow} H^1(G^1) \stackrel{\gamma_5}{\longrightarrow} H^1(N\otimes E) \to \cdots;$$on the other hand, tensoring 
\eqref{eq:N4fold} by $E$ and passing to cohomology, one gets also: 
\begin{equation}\label{eq:Flam2termag}
0 \to H^0(N \otimes E) \to W \otimes H^0(E) \stackrel{{\mu_E}_|}{\longrightarrow} H^0(E \otimes E) \stackrel{\psi}{\longrightarrow} H^1(N\otimes E) \to \cdots,
\end{equation}where 
${\mu_E}_| := {\mu_E}_{|_{W \otimes H^0(E)}}$ as $W \otimes H^0(E) \subseteq H^0(E)^{\otimes 2}$. Since $\mu_E$ is injective by assumptions in \eqref{eq:4fold}, 
${\mu_E}_|$ is injective too. Therefore one has $h^0(N \otimes E) = 0$, which reduces 
\eqref{eq:help}  and \eqref{eq:Flam2termag}, respectively, to 
\begin{equation}\label{eq:Flam2quatermag}
H^1(\bigwedge^2 N) \stackrel{\gamma_4}{\hookrightarrow} H^1(G^1) \stackrel{\gamma_5}{\longrightarrow} 
H^1(N\otimes E) \to \cdots \;\; {\rm and} \;\; 
W \otimes H^0(E) \stackrel{{\mu_E}_|}{\hookrightarrow} H^0(E \otimes E) \stackrel{\psi}{\longrightarrow} H^1(N\otimes E) \to \cdots.
\end{equation}

\begin{claim}\label{cl:Flam2bismag} The map $\gamma_5$ is injective.
\end{claim}

\begin{proof}[Proof of Claim \ref{cl:Flam2bismag}] From Claim \ref{cl:Flam2mag}, the map $\pi$ in \eqref{eq:Flam2mag} induces an isomorphism 
$$\frac{H^0(\bigwedge^2 E)}{{\rm Im} ({\lambda_E}_|)} = \frac{H^0(\bigwedge^2 E)}{\bigwedge^2 W} \stackrel{\overline{\pi} \; \cong}{\longrightarrow} H^1(G^1).$$Composing with 
$\gamma_5$, one gets 
$$\frac{H^0(\bigwedge^2 E)}{{\rm Im} ({\lambda_E}_|)} = \frac{H^0(\bigwedge^2 E)}{\bigwedge^2 W} \stackrel{\gamma_5 \circ \overline{\pi}}{\longrightarrow} H^1(N\otimes E)$$which 
is compatible with the injection 
$$\frac{H^0(E^{\otimes 2})}{{\rm Im} ({\mu_E}_|)} = \frac{H^0(E^{\otimes 2})}{W \otimes H^0(E)} \stackrel{\overline{\psi}}{\hookrightarrow} H^1(N\otimes E)$$induced by $\psi$ as in \eqref{eq:Flam2quatermag}. Since $W \subseteq H^0(E)$, write $H^0(E) = W \oplus U$ when $W \subsetneq H^0(E)$. Therefore 
$$W \otimes H^0(E) = W^{\otimes 2} \oplus W \otimes U = \bigwedge^2 W \oplus Sym^2(W) \oplus W \otimes U.$$Since 
$H^0(E^{\otimes 2}) = H^0 (\bigwedge^2 E) \oplus H^0(Sym^2(E))$ and $\mu_E$, $\lambda_E$ and $\sigma_E$ are injective, 
then $$ \frac{H^0(E^{\otimes 2})}{W \otimes H^0(E)} = \frac{H^0(\bigwedge^2 E)}{\bigwedge^2 W \oplus \left((W\otimes U) \cap \bigwedge^2 H^0(E)\right)} \oplus 
\frac{H^0(Sym^2(E))}{Sym^2(W) \oplus \left((W\otimes U) \cap Sym^2(H^0(E))\right)}.$$ The injectivity of $\overline{\psi}$ implies the injectivity of its restriction 
$$ \frac{H^0(\bigwedge^2 E)}{\bigwedge^2 W \oplus \left((W\otimes U) \cap \bigwedge^2 H^0(E)\right)} \stackrel{\overline{\psi}_|}{\hookrightarrow} H^1(N\otimes E).$$ If we prove that $\overline{\psi}_{|}$ coincides with $\gamma_5 \circ \overline{\pi}$, then $\gamma_5 \circ \overline{\pi}$ is therefore injective, so is $\gamma_5$. For these purposes, it suffices to show that $\bigwedge^2 W \oplus \left((W\otimes U) \cap \bigwedge^2 H^0(E)\right)$ is equal to $\bigwedge^2 W$. 

To do this, observe that in $H^0(E)\otimes H^0(E)$ the elements of $\bigwedge^2 H^0(E)$ correspond to skew-symmetric matrices with $h^0(E)$ rows and $h^0(E)$ columns. 
Since, as above,  $H^0(E) = W \oplus U$, such skew-symmetric matrices have the following type, namely:
\begin{equation}
\label{eq:tensormatrix1}
\left(
\begin{array}{cc}
A & B \\
C & D
\end{array}
\right),
\end{equation}
where $A^T=-A, D^T=-D, B^T+C=0$ and where $A$ is a square matrix with $r+4$ rows and $D$ is a square matrix with $h^0(E)-r-4$ rows. 

Let us now describe the elements of $W \otimes H^0(E)$ in $H^0(E) \otimes H^0(E)$; these correspond to square matrices of the following form, namely:
\begin{equation}
\label{eq:tensormatrix2}
\left(
\begin{array}{cc}
L & 0 \\
M & 0
\end{array}
\right),
\end{equation}
where $L$ is a square matrix with $r+4$ rows  and where $M$ is a matrix with $h^0(E)-r-4$ rows and $r+4$ columns. By \eqref{eq:tensormatrix1} and \eqref{eq:tensormatrix2}, 
an element in $\bigwedge^2 W \oplus \left((W\otimes U) \cap \bigwedge^2 H^0(E)\right)$ has to be a skew-symmetric matrix in $M_{r+4}({\mathbb C})$, namely an element in $\bigwedge^2W$. 
This implies that $\left((W\otimes U) \cap \bigwedge^2 H^0(E)\right) = \{0\}$, as claimed. 

Therefore, the map $\overline{\psi}_{|}$ coincides with the map $\gamma_5 \circ \overline{\pi}$, which is therefore injective. Since $\overline{\pi}$ is an isomorphism, one has that $\gamma_5$ is injective. 
\end{proof}

Claim \ref{cl:Flam2bismag} and \eqref{eq:Flam2quatermag} give $H^1(\bigwedge^2 N) = (0)$, proving the surjectivity of $H^1(\alpha_2)$ as in \eqref{eq:condequiv}, which 
completes the proof of Theorem \ref{thm:4fold}.  
\end{proof}


\subsection{Consequences and examples}\label{Ex4fold} In this section we discuss some direct consequences of Theorem \ref{thm:4fold}, showing how this main result can be related to several aspects like 
Fano varieties, Lazarsfeld-Mukai bundles, etcetera.

\vskip 7 pt

\noindent
(a) To start with, let $V$ be any smooth, projective variety and denote by $T_V$ its tangent bundle. As proved, for instance, in \cite[Proposition\;4.1]{hsiao}, if $T_V$ is nef and big, then $V$ is a Fano manifold, i.e., $\det(T_V)=-K_V$ \color{black}is ample. In loc. cit., the author poses Question 4.5., i.e., $$\mbox{\it If $V$ is Fano with nef tangent bundle $T_V$, is it true that $T_V$ is big?}.$$As explained in \cite[p.\,1550098-8]{hsiao}, the affirmative answer to the previous question has been proved up to dimension $3$. 

Theorem \ref{thm:4fold} allows us to answer this question in dimension four in some cases. More precisely, the following holds.

\begin{proposition}\label{prop:bignesstangent} Let $V$ be a smooth projective four-fold. Assume $V$ is a Fano manifold and $T_V$ is globally generated with $h^0(T_V) \geqslant \; 9$. Then $T_V$ is a big vector bundle on $V$.
\end{proposition}
\begin{proof} To prove the proposition, it suffices to verify conditions \eqref{eq:4fold}. 

Since $V$ is Fano, then $h^{p,0}(V)=0$ for $p \neq 0$, hence $q(V)=0$. Moreover, the following holds:
$$
H^i(V, \det(T_V)^{-1} \color{black})=H^i\left(V, K_V\right)\simeq H^{4-i}\left(V, {\mathcal O}_V \right)^{\vee}=\{0\}.
$$

As for $h^3\left(V, E^{\vee} \otimes \det(E)^{-1}\right)=0$,  with $E = T_V$, \color{black} we have
$$
H^3\left(V, T_V^{\vee} \otimes K_V \right) \simeq H^1\left(V, T_V\right)^{\vee}.
$$We want to show that this is zero. By \cite[Proposition\;2.1]{occ}, $V$ is a homogeneous variety as its tangent bundle $T_V$ is globally generated. As recalled, for instance, in loc. cit., Theorem 2.2., any homogeneous manifold is isomorphic to the product $A \times Y_1 \times \ldots \times Y_k$, where $A$ is an abelian variety whreas $Y_i$ is a rational homogeneous manifold, i.e., the quotient of a simple Lie group $G_i$ by \color{black} a parabolic subgroup $P_i$. Since $V$ is Fano, so $V$ is isomorphic to a product of rational homogeneous manifolds; more precisely, there  cannot \color{black} be an abelian factor in the product mentioned before, otherwise in that case one would have $q(V) \neq 0$, a contradiction. 
By \cite[Proposition\;11.6]{snow}, we have $H^1(G/P, T_{G/P})=\{0\}$ which implies $H^1(V, T_V)=\{0\}$ as desired; in particular also the third requirement in \eqref{eq:4fold} is fulfilled.

Finally, it remains to check the injectivity of the map $\mu_{T_V}$ as in \eqref{eq:muE} (with $E = T_V$). As recalled in \eqref{eq:injectiveaccazero}, this is equivalent to requiring that $H^0\left(V, M_{T_V} \otimes T_V\right) =\{0\}$. First, notice that the rank of $M_{T_V}$ is greater than $4$ as 
it follows from the assumption $h^0(T_V) \geqslant \; 9$  and the fact that, by \eqref{eq:mukailaz}, $rk(M_{T_V}) = h^0(V, T_V) - 4$. Now, suppose, by contradiction, that there exists a non-zero section 
$\sigma \in H^0\left(V, M_{T_V} \otimes T_V\right)$. This yields $${\mathcal O}_V \stackrel{\sigma}{\hookrightarrow}  M_{T_V} \otimes T_V,$$ which is equivalent to an injective map 
$$M_{T_V}^{\vee} \hookrightarrow T_V$$(cf. \cite[Prop.\,6.3(c), p. 234 and Prop. 6.7, p.235]{H}). This would imply that $M_{T_V}^{\vee}$ has rank less than or equal to $4$, which is a contradiction. Therefore, all the requirements in $\eqref{eq:4fold}$ are fulfilled and the proposition is completely proved.
\end{proof}

\begin{remark}
Proposition \ref{prop:bignesstangent} answers Question 4.5 in \cite{hsiao} when $T_V$ is globally generated and has at least $9$ global sections, as any globally generated bundle is nef. 
\end{remark}

\vskip 7 pt

\noindent
(b) In the same light of Example \ref{ExSurf}-(e), we want to show that Theorem \ref{thm:4fold} gives sufficient, but not necessary, cohomological conditions for bigness. We discuss here an example which has been inspired by \cite[Prop.\;2]{Beau}, concerning rank-two Lazarsfeld-Mukai vector bundles on suitable smooth, projective surfaces. Here we will instead consider suitable Lazarsfeld-Mukai vector bundles on smooth, projective four-folds. For simplicity in what follows, with a small abuse of notation, we will identify line bundles with associated Cartier divisors 
using interchangeably multiplicative and additive notation.

Let $V$ be a smooth, projective four-fold. We assume $q(V)=0$ and we take any ample, globally generated line bundle $L$ on $V$ such that 
$h^0(V, L) : = x \geqslant 4$. 

Let $Y$ be a general element in the linear system $|2L|$;  by Bertini's theorem, $Y$ is smooth. Moreover, the exact sequence 
\begin{equation}\label{eq:ref.20}
0 \rightarrow {\mathcal O}_V(-2L) \rightarrow {\mathcal O}_V \rightarrow {\mathcal O}_Y \rightarrow 0
\end{equation} ensures that $h^0({\mathcal O}_Y) = 1$; indeed, one has 
$h^0({\mathcal O}_V(-2L)) = 0 = h^1({\mathcal O}_V(-2L))$, the second equality following from the ampleness of $L$, 
Serre duality $h^1({\mathcal O}_V(-2L)) = h^3(K_V+ 2L) $ and the Kodaira vanishing theorem; thus $Y$, being smooth, is also irreducible.

Let $A := {\mathcal O}_Y(L)$. Since $L$ is globally generated, then $A$ is globally generated on $Y$. It therefore 
makes sense to consider the vector bundle $E$ on $V$, defined by the exact sequence: 
\begin{equation}\label{eq:Emuklaz}
0 \rightarrow E^{\vee} \rightarrow H^0(A) \otimes {\mathcal O}_V \overset{ev}{\rightarrow} A \rightarrow 0,
\end{equation} which is called the {\em Lazarsfeld-Mukai vector bundle} associated to the pair $(Y,A)$.

\begin{lemma}\label{lem:1} $E$ is big. 
\end{lemma}

\begin{proof} If one dualizes \eqref{eq:Emuklaz}, one gets 
\begin{equation}\label{eq:dualE1}
0 \rightarrow H^0(A)^{\vee} \otimes {\mathcal O}_V \rightarrow E \rightarrow {\mathcal Ext}^1(A, {\mathcal O}_V) \rightarrow 0;
\end{equation}the sheaf ${\mathcal Ext}^1(A, {\mathcal O}_V)$ is supported on $Y$ and 
${\mathcal Ext}^1(A, {\mathcal O}_V) \simeq {\mathcal Ext}^1({\mathcal O}_Y, {\mathcal O}_V) \otimes A^{-1}$. 
Similarly, dualizing \eqref{eq:ref.20}, we get
$$
0 \rightarrow {\mathcal O}_V \rightarrow {\mathcal Hom}({\mathcal O}_V(-2L), {\mathcal O}_V) \simeq {\mathcal O}_V(Y) \rightarrow {\mathcal Ext}^1({\mathcal O}_Y, {\mathcal O}_V) \rightarrow 0,
$$i.e. ${\mathcal Ext}^1({\mathcal O}_Y, {\mathcal O}_V) \simeq {\mathcal O}_Y(Y)$. Thus, \eqref{eq:dualE1} yields
\begin{equation}\label{eq:dualE2}
0 \rightarrow H^0(A)^{\vee} \otimes {\mathcal O}_V \rightarrow E \rightarrow A \rightarrow 0,
\end{equation}  since ${\mathcal Ext}^1(A, {\mathcal O}_V) \cong {\mathcal Ext}^1({\mathcal O}_Y, {\mathcal O}_V) \otimes A^{-1} \cong 
{\mathcal O}_Y(Y) \otimes A^{-1} \simeq {\mathcal O}_Y(2L) \otimes {\mathcal O}_Y(-L) \simeq {\mathcal O}_Y(L)=A$. 

Tensoring \eqref{eq:dualE2} by ${\mathcal O}_V (-L)$  gives 
$$0 \rightarrow H^0(A)^{\vee} \otimes {\mathcal O}_V (-L) \rightarrow E (-L) \rightarrow \mathcal O_Y \rightarrow 0,$$from which 
we deduce $h^0(E(-L)) = 1$, as $h^0({\mathcal O}_V (-L)) = 0 = h^1({\mathcal O}_V (-L))$ (the latter equality is a consequence of Serre duality, 
Kodaira vanishing theorem and the ampleness of $L$). Since $L$ is ample and $E(-L)$ is effective, by \cite[Example 6.1.23, p. 18]{Laz2} it follows that $E$ is big. 
\end{proof}

We want to show that, nonetheless, the pair $(V, E)$ as above satisfies all but one of the assumptions in Theorem \ref{thm:4fold}. 
As for the regularity of $V$, we have $q(V)=0$ by assumption. Moreover, the fact that $E$ is globally generated follows from \eqref{eq:dualE2} and the global generation of 
$A$. 

Let us show that $h^i((\det E)^{-1}) =0$, for any $i \in \{1, 2, 3\}$; if we tensor \eqref{eq:ref.20} by ${\mathcal O}_V(L)$, we get 
\begin{equation}\label{eq:ref.21}
0 \rightarrow {\mathcal O}_V(-L) \rightarrow {\mathcal O}_V (L)\rightarrow {\mathcal O}_Y(L) = A \rightarrow 0.
\end{equation} Combining \eqref{eq:Emuklaz} and \eqref{eq:ref.21}, we get $c_1(E)=2L$.
 Therefore,  
$$
h^i((\det E)^{-1})=h^i(-2L)=h^{4-i}(K_V+2L)=0, \qquad i \in \{1, 2, 3\},
$$ the latter equality due to the Kodaira vanishing theorem and the ampleness of $L$.

 As for condition $h^3(E^{\vee} \otimes (\det E)^{-1})=0$, if we 
tensor \eqref{eq:Emuklaz} by ${\mathcal O}_V(-2L)$, we get the following exact sequence:
$$
0 \rightarrow E^{\vee} \otimes {\mathcal O}_V(-2L) \rightarrow H^0(A) \otimes {\mathcal O}_V(-2L) \rightarrow A \otimes {\mathcal O}_V(-2L) \rightarrow 0.
$$
By the definition of $A$ and the fact $c_1(E) = 2L$, we have therefore 
$$
0 \rightarrow E^{\vee} \otimes (\det E)^{-1} \rightarrow H^0(A) \otimes {\mathcal O}_V(-2L) \rightarrow {\mathcal O}_Y(-L) \rightarrow 0;
$$in order to show that $h^3(E^{\vee} \otimes (\det E)^{-1}) =0$, it suffices to prove that 
$h^2({\mathcal O}_Y(-L)) = 0 = h^3({\mathcal O_V}(-2L))$. \color{black} Notice indeed that $h^2({\mathcal O}_Y(-L))=h^1(K_Y+L_{|Y})=0$, as it follows from Serre duality, 
the Kodaira vanishing theorem and the ampleness of $L_{|Y} = A $ on $Y$. Similarly,  
$h^3({\mathcal O}_V(-2L))= h^1(K_V+2L)=0$. 

From \eqref{eq:dualE2}, one has $rk (E) = h^0(Y,A)$; moreover, since $q(V) = 0$, \eqref{eq:dualE2} also gives $h^0(V,E) = 2 h^0(Y,A) = 2 rk (E)$. 
From \eqref{eq:ref.21} and $h^0({\mathcal O}_V(-L)) = 0 = h^1({\mathcal O}_V(-L))$, it follows that $h^0(Y,A) = h^0(V, L) = x$ so  
$rk(E) = x$ and $h^0(E) = 2x$. Since $x \geqslant 4$ by assumption, one has that also condition $h^0(E) \geqslant rk(E) + 4$ is satisfied.

Finally, notice that, from \eqref{eq:mukailaz} and the fact that $h^0(V,E) = 2 rk (E)$, one has $rk (M_E) = rk (E)$; we want to show more precisely 
that 
\begin{equation}\label{eq:dualME}
M_E \simeq E^{\vee}.
\end{equation}If the previous isomorphism holds true, then we will have 
$$h^0(M_E \otimes E) = h^0(E^{\vee} \otimes E) = h^0({\mathcal End} (E)) \geqslant 1$$(as, for any vector bundle, one has 
$\mathbb{C}^* \subseteq H^0({\mathcal End} (E))$) so from \eqref{eq:injectiveaccazero}, $\mu_E$ will be be not injective.

To prove \eqref{eq:dualME}, from the cohomology sequence associated to \eqref{eq:dualE2} and $q(V) = 0$, we get 
$$0 \to H^0(A)^{\vee} \to H^0(E) \to H^0(A) \to 0.$$Plugging it within \eqref{eq:mukailaz}, \eqref{eq:Emuklaz} and \eqref{eq:dualE2}, we get the following commutative diagram

\begin{displaymath}
\begin{array}{ccccccc}
      &             &     &     0                       &     &     0     &     \\
       &             &     &     \downarrow                        &     &     \downarrow      &     \\
  &  &  & H^0(A)^{\vee} \otimes \mathcal O_V & =  & H^0(A)^{\vee} \otimes \mathcal O_V &   \\
      &            &     &     \downarrow                        &     &     \downarrow      &     \\
0 \to & M_{E} & \to & H^0(E) \otimes \mathcal O_V& \longrightarrow & E & \to 0  \\
     &             &     &     \downarrow                        &     &     \downarrow      &     \\
0 \to  & E^{\vee} & \to & H^0(A) \otimes \mathcal O_V & \longrightarrow  & A & \to 0  \\
          &             &     &     \downarrow                        &     &     \downarrow      &     \\
     &             &     &     0                       &     &     0     &
\end{array}
\end{displaymath}which proves \eqref{eq:dualME}. 

\color{black}


\end{document}